\documentclass[hoptionsi]{aomart}

\makeatletter 
\fancypagestyle{firstpage}{%
  \fancyhf{}%
  \if@aom@manuscript@mode
    \lhead{\begin{picture}(0,0)%
        \put(-21,-25){\usebox{\@aom@linecount}}%
      \end{picture}}
  \fi
  \chead{}%
   \cfoot{\footnotesize\thepage}}%
\doinumber{}
\makeatother

\usepackage{amsfonts,amssymb,latexsym,xspace,epsfig,graphics,color}
\usepackage{amsmath,enumerate,stmaryrd,xy}
\usepackage{bbm}
\usepackage{subfigure}
\usepackage{dsfont}
\usepackage{algorithmic}
\usepackage{algorithm}
\usepackage{stmaryrd,ushort,dsfont,mathrsfs}

\newcommand{\E}{\mathbb{E}}

\newcommand{\eps}{\varepsilon}

\newcommand{\N}{\ensuremath{\mathbb{N}}}

\newcommand{\R}{\ensuremath{\mathbb{R}}}     
\newcommand{\Z}{\ensuremath{\mathbb{Z}}}    
\renewcommand{\P}{\ensuremath{\mathbb{P}}}

\newcommand{\X}{\ensuremath{\mathcal{X}}}

\newcommand\smallO{
  \mathchoice
    {{\scriptstyle\mathcal{O}}}
    {{\scriptstyle\mathcal{O}}}
    {{\scriptscriptstyle\mathcal{O}}}
    {\scalebox{.5}{$\scriptscriptstyle\mathcal{O}$}}
  }

\DeclareMathOperator{\var}{\mathrm{var}}

\newtheorem{remark}{Remark}
\newtheorem{lemma}{Lemma}

\newtheorem{theo}{Theorem}
\newtheorem{prop}{Proposition}
\newtheorem{coro}{Corollary}

\newtheorem*{assump}{Assumption}

\DeclareGraphicsExtensions{.pdf,.gif,.jpg}

\title[Mixing rates for potentials of non-summable variations]{Mixing rates for potentials of non-summable variations}

\author{Christophe Gallesco$^1$}
\address{$^1$Departmento de Estat\'istica, Instituto de Matem\'atica, Estat\'istica e Ci\^encia de Computa\c{c}\~ao, Universidade de Campinas, Brasil.}
\email{gallesco@unicamp.br}
\author{Daniel Y. Takahashi$^2$}
\address{$^2$ Instituto do C\'erebro, 
Universidade Federal do Rio Grande do Norte, Brasil.} 
\email{takahashiyd@gmail.com}

\keyword{g-measure, Chains of infinite order, coupling, concentration inequality, Poisson autoregression model}
\subjclass[2000]{Primary 37A25, 37A50  Secondary 60G07.\\
{ Keywords:} $g$-measure, chains of infinite order, coupling, Central Limit Theorem, concentration inequality}

\copyrightnote{}

\begin{document}

\begin{abstract}
Mixing rates, relaxation rates, and decay of correlations for dynamics defined by potentials with summable variations are well understood, but little is known for non-summable variations. This paper exhibits upper bounds for these quantities for dynamics defined by potentials with square summable variations. We obtain these bounds as corollaries of a new block coupling inequality between pair of dynamics starting with different histories. As applications of our results, we prove a new weak invariance principle and a Hoeffding-type inequality.
\end{abstract}

\maketitle


\section{Introduction}
Let $A$ be a countable set, called {\it alphabet}. Consider a measurable function $\phi: A\times A^{\{-1,-2, \ldots \}}\to \R$ such that $\sum_{a\in A}e^{\phi(a,x)}=1$ for all $x\in A^{\{-1,-2, \ldots \}}$. The function $\phi$ is called a {\it normalized potential}, and the probability kernel $g:= e^{\phi}$ (also known as {\it $g$-function}) is a natural generalization of Markov kernels.
Let $\eta=(\eta_n)_{\Z}$ be the canonical projections on $A^{\Z}$, \emph{i.e.}, for all $x\in A^{\Z}$ and all $n\in \Z$, $\eta_n(x)=x_n$. For $y \in A^{\{0,-1, \ldots \}}$, let $\mu^y$ be the probability measure on $A^{\Z}$ such that $\mu^y[(\eta_0,\eta_{-1},\dots) \in B]=\delta_{y} [B]$ for all $B\subset A^{\{0, -1, \ldots \}}$ measurable. For $n\geq 0$, $\mu^y[\eta_{n+1}=a\mid (\eta_n,\eta_{n-1},\dots)=x]=e^{\phi(a,x)}$ 
for every $a\in A$ and $\mu^y$-a.e.\ $x$ in $A^{\{-1,-2, \ldots \}}$. Let $T$ denotes the shift operator on $A^{\{0,-1, \ldots \}}$. We indicate by $d_{\text{TV}}$ the total variation distance, that is, if $P$ and $Q$ are two probability measures on the same $\sigma$-algebra $\mathcal{F}$,
\[
d_{\text{TV}}(P,Q)=\sup_{F\in \mathcal{F}}|P[F]-Q[F]|.
\]

In this paper, we obtain upper bounds, respectively, for the relaxation rate 
\begin{equation*}
L(n) :=\sup_{y,z} d_{\text{TV}}(\mu^y[\eta_n \in \cdot\;],\mu^z[\eta_n \in \cdot\;]),
\end{equation*}
the mixing rate
\begin{equation*}
M(n) := \sup_{y,z} d_{\text{TV}}(\mu^y[(\eta_j)_{j\geq n} \in \cdot\;],\mu^z[(\eta_j)_{j\geq n} \in \cdot\;]),
\end{equation*}
and the decay rate of correlations
\begin{equation*}
\rho_{f,\hat{f}}(n) := \Big| \int f\circ T^n \;{\hat f} d\tilde{\mu}-\int f d\tilde{\mu}\int {\hat f} d\tilde{\mu} \Big|
\end{equation*}
when $\tilde{\mu}$ is the unique shift-invariant measure compatible with $\phi$ (see the next section for the definition of {\it compatibility}) and $f, \hat{f}$ are suitable functions (see Theorem \ref{Maintheo2}). \cite{bressaud/fernandez/galves/1999a} and  \cite{pollicott2000rates} obtained upper bounds for $L(n), M(n)$ and $\rho_{f,\hat{f}}(n)$  for potentials of summable variations and finite alphabets. \cite{gouezel2004sharp} obtained sharp lower bounds for the decay of correlation for dynamics with H\"older continuous (\emph{i.e.}, exponentially decaying) potentials and countable alphabet. Our contribution is twofold. We obtain upper bounds for $L(n), M(n),$ and $\rho_{f,\hat{f}}(n)$  when the variation rate $ \var_k(\phi)$ decays as $\mathcal{O}(k^{-(1/2+\delta')})$ for any $ \delta' > 0$. Moreover, our results also holds for countably infinite alphabet $A$. Theorem~\ref{Maintheo} is our main result, showing a new upper bound for the coupling error between $\mu^y$ and $\mu^z$. Corollary \ref{Cor1} answers a question posed in \cite{johansson2008square}, in which the authors ask for a bound for $L(n)$ when the variation of $\phi$ is not summable. Corollary \ref{Cor2} shows a bound for $M(n)$, which  cannot be achieved by simply using the union bound and Corollary \ref{Cor1}. The result is new even for the case of summable variations. The interest in $L(n)$ stands from the fact that it is the natural generalization of mixing times for Markov chains. \cite{GGT2018} showed that $M(n)$ converges to $0$ only when $\var_k(\phi)$ is square summable $\tilde{\mu}$-a.s., hence Corollary \ref{Cor2} covers the main cases of interest. Theorem \ref{Maintheo2} gives an upper bound for the speed of decay of correlations, extending Theorem 1 in \cite{bressaud/fernandez/galves/1999a}. \cite{johansson/oberg/pollicott/2012} showed, when the alphabet $A$ is finite, that there is a unique shift-invariant measure $\tilde{\mu}$ compatible with $\phi$ when $\var_k(\phi)  \in \smallO(k^{-(1/2)})$. Moreover,  \cite{berger2018non} showed that whenever $\var_k(\phi)  \in \mathcal{O}(k^{-(1/2-\delta)})$ for any $\delta > 0$, there exists a normalized potential $\phi$ that exhibits multiple compatible shift-invariant measures. Therefore, Theorem \ref{Maintheo2} also covers the main variation rates of interest under uniqueness of the compatible shift-invariant measure. In Corollary \ref{coro3}, we use  Theorem \ref{Maintheo2} to obtain upper bounds on rate of correlation decay for non-normalized potentials. We illustrate the application of our inequalities in three cases. The first application proves a novel weak invariance principle for additive functionals of dynamics with non-summable variations. The second application shows that we can obtain Hoeffding-type bounds for averages of random variables when the variation of $\phi$ is not summable. The third example illustrates how we can apply our results on a Poisson autoregression model, which is popular in applied works.

The proof technique is based on a renewal equation and coupling inequalities. These ideas were developed in \cite{coelho_quas_1998, bressaud/fernandez/galves/1999a, comets_fernandez_ferrari_2002}. We improve on the coupling bounds obtained in \cite{bressaud/fernandez/galves/1999a} by using a coupling between blocks of coordinates, instead of one coordinate at a time.  A block coupling idea was used in \cite{johansson/oberg/pollicott/2012} to obtain sharp conditions for uniqueness of the equilibrium measure for $\phi$ on a finite alphabet $A$, but no mixing rate was obtained. A difference between \cite{johansson/oberg/pollicott/2012} and our approach is that we upper bound the block coupling using different renewal processes leading to a distinct renewal equation. This new renewal equation allows us to upper bound the speed of decay of coupling inequality even when the variation is not summable (see Theorem \ref{Maintheo}).

\section{Definitions}
Let the alphabet $A$ be a countable set, $\mathcal{X}=A^{\Z}$ and $\mathcal{X}_{-} = A^{\Z_-}$ where $\Z_-=\{0,-1,-2,\dots\}$. We endow $\mathcal{X}$ and $\mathcal{X}_-$ with the product topology and its corresponding Borel $\sigma$-algebra. The topologies and $\sigma$-algebras considered on subsets of $\mathcal{X}$ and $\mathcal{X}_-$ will always be the trace topologies and $\sigma$-algebras. We denote by $x_i$ the $i$-th coordinate of $x \in \mathcal{X}$ and for $-\infty< i \leq j<\infty$ we write $x^{-i}_{-j}:=(x_{-i},\ldots, x_{-j})$, $x^{-i}_{-\infty}:=(x_{-i},x_{-i-1},\dots)$ and $x_i^{\infty}:=(\dots,x_{i+1},x_{i})$. If $i < j$, $x^i_j = \o$. For $x\in \mathcal{X}$ and $y \in \mathcal{X}_{-}$, a \emph{concatenation} $x^{0}_{-i}y$ is a new sequence $z\in \mathcal{X}_{-}$ with $z^{0}_{-i} = x^{-1}_{-i}$ and $z^{-i-1}_{-\infty} = y$. We take $\o$ to be the neutral element of the concatenation operation, that is $\o x=x$ for all $x\in \mathcal{X}_{-}$.  Note that we are using the convention, consistent with the concatenation operation, that when we scan an element $x\in\mathcal{X}$ from the left to the right we go further into the past.

Consider a measurable function $\phi: \mathcal{X}_{-}\to \R$, which we call potential.  We say that $\phi$ is {\it normalized} if it satisfies
\begin{equation*}
\sum_{a\in A}e^{\phi(ax)}=1
\end{equation*}
for all $x\in \mathcal{X}_{-}$. To a normalized potential $\phi$ we can associate a {\it probability kernel} $g$ on the alphabet $A$ by defining $g=e^{\phi}$. The \emph{variation} of order $k\geq 0$ of $\phi$ is defined by
\begin{equation*}
\var_{k}(\phi):=\frac{1}{2}\sup_{z \in \mathcal{X}}\sup_{x,y\in \mathcal{X}_-}\sum_{b\in A}\big|\phi(bz_{-k}^{-1}x)-\phi(bz_{-k}^{-1}y)\big|. 
\end{equation*}
When $A$ is finite, the variation is usually defined by taking the supremum over $b \in A$ instead of the sum. Nevertheless, our definition is more convenient when the alphabet is infinite and has appeared in the literature before \cite{chazottes2020optimal}. The constant $1/2$ in the definition relates $\var_k(e^{\phi})$ to total variation distance when $\phi$ is normalized.

We also define, for $k\geq 0$, the {\it $\chi^2$-variation} of order $k$ of $\phi$ as
\begin{equation*}
\chi^2_{k}(\phi)=\sup_{z\in \mathcal{X}}\sup_{x,y\in \mathcal{X}_-}\sum_{b\in A}\frac{\big(e^{\phi(bz_{-k}^{-1}x)}-e^{\phi(bz_{-k}^{-1}y)}\big)^2}{e^{\phi(bz_{-k}^{-1}y)}}. 
\end{equation*}

The use of $\chi^2$-variation to measure the regularity of potentials seems to be new, therefore it is interesting to compare it to variation, which is more standard. When $\phi$ is normalized, by using Cauchy-Schwarz inequality, we have that $ \var^2_{k}(e^\phi)\leq  \frac{1}{4}\chi^2_{k}(\phi)$ for any $k\geq 1$. When the alphabet  $A$ is finite and $\phi$ is normalized $\var^2_k(\phi)$ and $\chi^2_k(\phi)$ are comparable, that is, there exist positive constants $K_1$ and $K_2$ such that $K_1\var^2_k(\phi)\leq \chi^2_k(\phi)\leq  K_2\var^2_k(\phi)$. The $\chi^2$-variation introduced in this work will be particularly useful to study asymptotic properties of positive probability kernels on infinite $A$ (cf. Section \ref{Poisson}). 

Let $\eta=(\eta_n)_{\Z}$ be the canonical projections on $\mathcal{X}$, that is, for all $x\in \mathcal{X}$, $\eta_n(x)=x_n$ for all $n\in \Z$. We say that a probability measure $\mu$ on $\mathcal{X}$ is \emph{compatible} with a normalized potential $\phi$ if there exits a probability measure $P$ on $\mathcal{X}_-$ such that 
\[
\mu[\eta_{-\infty}^{0}\in B]=P[B]
\]
for all $B\subset \mathcal{X}_-$ measurable and if for $n\geq 0$
\begin{equation*}\label{compa}
\mu[\eta_{n+1}=a\mid \eta_{-\infty}^{n}=x]=e^{\phi(ax)}
\end{equation*}
for every $a\in A$ and $\mu$-a.e.\ $x$ in $\mathcal{X}_{-}$. \cite{johansson/oberg/pollicot/2007} showed that if
\begin{equation*}
\sum_{k = 0}^\infty \sup_{z \in \mathcal{X}}\sup_{x,y\in \mathcal{X}_-}\sum_{b\in A}\big(e^{\phi(bz_{-k}^{-1}x)/2}-e^{\phi(bz_{-k}^{-1}y)/2}\big)^2 < \infty,
\end{equation*}
then there is at most one shift-invariant invariant compatible measure with $\phi$. From \cite[Lemma 3.3.9. ]{reiss2012approximate}, for $k \geq 0$, we have that
\[
\sum_{b\in A}\big(e^{\phi(bz_{-k}^{-1}x)/2}-e^{\phi(bz_{-k}^{-1}y)/2}\big)^2  \leq \sum_{b\in A}\frac{\big(e^{\phi(bz_{-k}^{-1}x)}-e^{\phi(bz_{-k}^{-1}y)}\big)^2}{e^{\phi(bz_{-k}^{-1}y)}},\] 
hence the summability of $\chi^2_k(\phi)$ implies the existence of at most one shift-invariant invariant compatible measure.  

 When $\phi$ is not normalized, the definition of a compatible measure loses its meaning. Nevertheless, we can associate  a set of shift-invariant measures called equilibrium states for not necessarily normalized $\phi$  \cite{walters/1975}. Equilibrium states are characterized via {\color{red}a} variational principle and coincide with shift-invariant compatible measures when $\phi$ is normalized. An equilibrium state $\tilde{\mu}$ compatible with a normalized $\phi$ is also called $g$-measure in ergodic theory \cite{keane/1972}.  In probability literature, $g$-measures are known  as chains of complete connections \cite{doeblin/fortet/1937, iosifescu/grigorescu/1990}, chains of infinite order \cite{harris/1955}, \cite{keane/1972}, random-step Markov process \cite{kalikow}, and uniform martingale \cite{kalikow}. Compatible measures that are not necessarily shift-invariant are called $g$-chains \cite{johansson/oberg/pollicott/2012} or  stochastic chains of unbounded memory \cite{GGT2018}. When there is more than one shift-invariant measure compatible with $\phi$, we say that there is a \emph{phase transition}, otherwise we say that the shift-invariant compatible measure is \emph{unique}.

\section{Results}
\label{Results}

 In this paper, we will work under the following assumption.
 \begin{assump}[$\mathcal{A}$]
 \label{Assump1}
 $\phi$ is a potential on $\mathcal{X}_-$ such that for all $k\geq 1$, 
 \begin{equation} \label{eq:assumption}
 \chi^2_{k}(\phi)\leq \frac{C}{k^{1+\delta}}
 \end{equation}
for some $C>0$ and $\delta>0$.
 \end{assump}

\begin{remark}
When the alphabet $A$ is finite and $\phi$ is normalized, Assumption ($\mathcal{A}$) is equivalent to 
 \[
 \var_{k}(\phi)\leq \frac{C'}{k^{\frac{1+ \delta}{2}}},
 \]
for some $C' > 0$ and same $\delta $ as in \eqref{eq:assumption}. Observe that $ \var_k(\phi)$ is not summable when $\delta \in (0, 1]$.
\end{remark}  

Now, consider $\mathcal{X}\times \mathcal{X}$ with the projection maps $\hat{\eta}=(\hat{\eta}_n)_{n\in \Z}$ and $\hat{\omega}=(\hat{\omega}_n)_{n\in \Z}$ such that for $(x,y)\in \mathcal{X}\times \mathcal{X}$, $\hat{\eta}_n(x,y)=x_n$ and $\hat{\omega}_n(x,y)=y_n$ for all $n\in \Z$. Let us also denote by $\widehat{\mathcal{C}}(\phi)$ the set of probability measures $P$ on $\mathcal{X}\times \mathcal{X}$ such that the pushforward measures $\hat{\eta}_*P$ and $\hat{\omega}_*P$ are compatible with $\phi$.
We also introduce the process $X=(X_n)_{n\geq 1}$, such that for all $n\geq 1$,
\begin{equation*}
X_n=\mathds{1}\{\exists j\in [K_n, K_{n+1}), \hat{\eta}_j \neq \hat{\omega}_j\},
\end{equation*}
where $(K_n)_{n\geq 1}$ is a fixed strictly increasing sequence of natural numbers such that $K_1=1$. 
Here is our main result followed by two corollaries.
\begin{theo}
\label{Maintheo}
Let $\phi$ be a normalized potential that satisfies Assumption ($\mathcal{A}$). Let $K_n=\lfloor n^{\beta} \rfloor$ for $\beta\geq 1$ and $\beta>1/\delta$. For all measures $\mu$ and $\nu$ compatible with $\phi$, there exists $\P\in \widehat{\mathcal{C}}(\phi)$ such that $\hat{\eta}_*\P=\mu$, $\hat{\omega}_*\P=\nu$ and for  $n\geq 1$, 
\begin{equation*}
\P[X_n=1]\leq \frac{C_1}{n^{\frac{\beta \delta+1}{2}}}
\end{equation*}
where $C_1$ is a positive constant depending on $C, \delta$, and $\beta$.
\end{theo}

\begin{coro}
\label{Cor1}
Let $\phi$ be a normalized potential that satisfies Assumption ($\mathcal{A}$). If $\delta>1$, we have for all $n\geq 1$
\begin{equation*}
L(n)\leq \frac{C_2}{n^{\frac{1+\delta}{2}}}
\end{equation*}
where $C_2$ is a positive constant depending on $C$ and $\delta$.\\
$ $

If $\delta\in (0,1]$, we have for all $n\geq 1$ and $\delta'<\delta$, 
\begin{equation*}
 L(n)\leq \frac{C_3}{n^{\delta'}},
\end{equation*}
where $C_3$ is a positive constant that depends on $C$, $\delta$, and $\delta'$.
\end{coro}

\begin{coro}
\label{Cor2}
Let $\phi$ be a normalized potential that satisfies Assumption ($\mathcal{A}$). For all $\delta'<\delta$, we have for all $n\geq 1$
\begin{equation*}
M(n) \leq \frac{C_4}{n^{\frac{\delta'}{2}}},
\end{equation*}
where $C_4$ is a positive constant that depends on $C$, $\delta$, and $\delta'$.\\
\end{coro}

\begin{remark}
When $\delta>1$ and $A$ is finite, we can use \cite[Theorem 1]{bressaud/fernandez/galves/1999a} and the union bound to obtain
\begin{equation*}
M(n) \leq \frac{C_5}{n^{\frac{\delta-1}{2}}},
\end{equation*}
where $C_5 > 0$  is a constant that depends on $C$ and $\delta$. Hence, the result in Corollary \ref{Cor2} gives a sharper upperbound, even when the potential is summable and the alphabet $A$ is finite.
\end{remark}

We now look at the correlations decay for the shift-invariant measure compatible with a potential $\phi$. For this, we need the following definitions. Consider the shift operator $T: \mathcal{X}_{-}\to \mathcal{X}_{-}$ such that for all $x\in  \mathcal{X}_{-}$, $Tx=Tx_{-\infty}^0=x_{-\infty}^{-1}$. For non-constant $\phi$, let us consider the seminorm
\[
\|f\|_{\phi}=\sup_{k\geq 1}\frac{\var_k(f)}{\var_k(e^\phi)},
\]
and the subspace of $\mathcal{C}(\mathcal{X}_-,\R)$ defined by
$$V_{\phi}=\Big\{f \in \mathcal{C}(\mathcal{X}_-,\R):\|f\|_{\phi}<\infty   \Big\}.$$
\begin{theo}
\label{Maintheo2}
Let $\phi$ be a normalized potential that satisfies Assumption ($\mathcal{A}$). Assume that a shift-invariant probability measure $\tilde{\mu}$ compatible with $\phi$ exists. Let $f\in L^1(\tilde{\mu})$ and $\hat{f}\in V_{\phi}$.

If $\delta>1$, we have for all $n\geq 1$
\begin{equation*}
\rho_{f,\hat{f}}(n) \leq \frac{C_6}{n^{\frac{1+\delta}{2}}}\|f\|_1 \|\hat{f}\|_{\phi}
\end{equation*}
where $C_6$ is a positive constant that depends on $C$ and $\delta$.

If $\delta\in (0,1]$, we have for all $n\geq 1$ and $\delta'<\delta$, 
\begin{equation*}
\rho_{f,\hat{f}}(n) \leq \frac{C_7}{n^{\delta'}}\|f\|_1 \|\hat{f}\|_{\phi},
\end{equation*}
where $C_7$ is a positive constant that depends on $C$, $\delta$, and $\delta'$.
\end{theo}

\begin{remark}
When $\delta>1$ and $A$ is finite, Theorem \ref{Maintheo2} recovers the rate obtained in \cite[Theorem 1]{bressaud/fernandez/galves/1999a}.
\end{remark}

\begin{remark}
When $A$ is finite, continuity of $\phi$ guarantees the existence of a compatible shift-invariant measure, therefore the assumption on the existence of a compatible measure in Theorem \ref{Maintheo2} is redundant. When $A$ is infinite, the existence of a shift-invariant compatible measure is not immediate. Sufficient conditions for existence of shift-invariant compatible measures when $A$ is infinite are given in \cite{fernandez/maillard/2005, johansson/oberg/pollicot/2007}. See Section \ref{Poisson} for a concrete example. Whenever a shift-invariant compatible measure exists, Assumption ($\mathcal{A}$) implies uniqueness of $\tilde{\mu}$ in Theorem \ref{Maintheo2} \cite{johansson/oberg/pollicot/2007}, although uniqueness is not a priori necessary for Theorem \ref{Maintheo2}. 
\end{remark}

A natural question is whether we can obtain an upper bound for the rate of correlation decay for a potential $\phi$ that is not normalized. When $A$ is finite, we can use the same strategy as in \cite{walters/1975, bressaud/fernandez/galves/1999a, pollicott2000rates}. The idea is to study normalized potentials $\psi$ that are cohomologous to $\phi$, \it{i.e.}, $\psi = \phi + h - h\circ T + c$, for some $h \in \mathcal{C}(\mathcal{X}_-,\R)$ and $c \in \R$. If $\phi$ and $\psi$ are cohomologous, then both functions have the same associated equilibrium states \cite{walters/1975}. Hence, properties of equilibrium states for $\phi$ can be obtained by studying shift-invariant measures compatible with $\psi$.  \cite{walters/1975} proved that when the rate of variation of $\phi$ is summable, there exist a unique $h$ and unique $c$ such that $\psi$ is normalized potential. Moreover, from the construction of $h$ in \cite{walters/1975}, we have that $\var_{k}(h) \leq \sum_{j \geq k}\var_{j}(\phi)$. This implies that $\var_{k}(\psi) \leq 3\sum_{j \geq k}\var_{j}(\phi)$. Using these results, we obtain the following corollary, which improves the results in \cite{bressaud/fernandez/galves/1999a, pollicott2000rates}.

\begin{coro}
\label{coro3}
Let the alphabet $A$ be finite and $\phi$ be a potential not necessary normalized. Assume there exist a constant $C > 0$ and $\delta > 0$ such that 
 \[
 \var_{k}(\phi)\leq \frac{C}{k^{\frac{3+ \delta}{2}}}.
 \]
Let $\tilde{\mu}$ be an equilibrium state for $\phi$, $f\in L^1(\tilde{\mu})$ and $\hat{f}\in V_{\phi}$.

If $\delta>1$, we have for all $n\geq 1$
\begin{equation*}
\rho_{f,\hat{f}}(n) \leq \frac{C_8}{n^{\frac{1+\delta}{2}}}\|f\|_1 \|\hat{f}\|_{\phi}
\end{equation*}
where $C_8$ is a positive constant that depends on $C$ and $\delta$.

If $\delta\in (0,1]$, we have for all $n\geq 1$ and $\delta'<\delta$, 
\begin{equation*}
\rho_{f,\hat{f}}(n) \leq \frac{C_9}{n^{\delta'}}\|f\|_1 \|\hat{f}\|_{\phi},
\end{equation*}
where $C_9$ is a positive constant that depends on $C$, $\delta$, and $\delta'$.
\end{coro}

\begin{remark}
When $\delta > 1$, Corollary \ref{coro3} recovers the rate obtained in  \cite[Theorem 1 (1)]{pollicott2000rates}. To generalize Corollary \ref{coro3} to infinite alphabet, we need a result equivalent to \cite[Theorem 3.3]{walters/1975} for infinite alphabet, which is currently unavailable. 
\end{remark}

\section{Technical lemmas}
Here we collect some results that we will use to prove Theorem \ref{Maintheo}. We first recall the definitions of the Kullback-Leibler and Pearson $\chi^2$ divergences. Let $P$ and $Q$ be two probabilities on some discrete space $\mathcal{Y}$.
\[
D_{\text{KL}}(P||Q)=\sum_{y\in \mathcal{Y}}P(y)\ln\Big(\frac{P(y)}{Q(y)}\Big)
\]
and
\[
D_{\chi^2}(P||Q)=\sum_{y\in \mathcal{Y}}\frac{(P(y)-Q(y))^2}{Q(y)}.
\]
It is well known that $D_{\text{KL}}(P||Q)\leq D_{\chi^2}(P||Q)$ (cf. \cite[eq. 5]{sason2016f}).
\begin{lemma}
\label{chideu}
Let $x, y\in \mathcal{X}_-$ and $\mu,\nu \in \mathcal{P}(\phi)$ such that $\mu[\eta_{-\infty}^{0}\in\cdot\;]=\delta_x(\cdot)$ and\\ $\nu[\eta_{-\infty}^{0}\in\cdot\;]=\delta_y(\cdot)$.  For all $n\geq 1$, $0\leq k\leq n-1$ and all $a, b, c\in \mathcal{X}$, we have that 
\begin{align}
\label{Div}
D_{\emph{KL}}\Big(\mu\Big[\eta_{K_{n}}^{K_{n+1}-1}\in\; \cdot\; \Big | \eta_1^{K_n-1}= a_{K_{n-k}}^{K_n-1} b_{1}^{K_{n-k}-1}\Big]\nonumber\\
\Big|\Big| \nu \Big[\eta_{K_{n}}^{K_{n+1}-1}\in \;\cdot \;\Big | &\eta_1^{K_n-1}=a_{K_{n-k}}^{K_n-1}c_{1}^{K_{n-k}-1}\Big]\Big)\nonumber\\
\leq \sum_{j=K_n}^{K_{n+1}-1}\chi^2_{j-K_{n-k}}(\phi).
\end{align}
\end{lemma}

\noindent
{\it Proof.} Let us simply denote by $D$ the left-hand term of inequality (\ref{Div}). We have by the chain rule property of the Kullback-Leibler divergence \cite[Theorem 2.5.3.]{cover/2006}
\begin{align*}
D
=&\sum_{i=K_n}^{K_{n+1}-1}D_{\text{KL}}\Big( \mu\Big[\eta_i\in \cdot \;\Big | \eta_1^{i-1}=z_{K_n}^{i-1}a_{K_{n-k}}^{K_n-1} b_{1}^{K_{n-k}-1}\Big]  \nonumber\\
&\phantom{******************}\Big|\Big| \nu\Big[\eta_i\in \cdot \;\Big | \eta_1^{i-1}=z_{K_n}^{i-1}a_{KK_{n-k}}^{K_n-1}c_{1}^{K_{n-k}-1}\Big]    \Big)\nonumber\\
=&:\sum_{i=K_n}^{K_{n+1}-1}D_i.
\end{align*}
Then, we use the well known bound,
\begin{align*}
D_i&\leq D_{\chi^2}\Big(\mu\Big[\eta_i\in \cdot \;\Big | \eta_1^{i-1}=z_{K_n}^{i-1}a_{K_{n-k}}^{K_n-1} b_{1}^{K_{n-k}-1}\Big]\nonumber\\
&\phantom{***************} \Big|\Big| \nu\Big[\eta_i\in \cdot \;\Big | \eta_1^{i-1}=z_{K_n}^{i-1}a_{K_{n-k}}^{K_n-1}c_{1}^{K_{n-k}-1}\Big]\Big)\nonumber\\
&\leq \chi^2_{i-K_n}(\phi)
\end{align*}
to conclude the proof.
\qed

\begin{lemma}
\label{Lemalg}
For $\alpha> 1$ and $0<a<b$, we have that
\begin{equation}
\label{Alg}
\frac{(b+1)^{\alpha}-a^{\alpha}}{b^{\alpha}-a^{\alpha}}\geq \frac{(b+1)^{\alpha-1}-a^{\alpha-1}}{b^{\alpha-1}-a^{\alpha-1}}.
\end{equation}
\end{lemma}
\noindent
{\it Proof.} By algebraic computations, we obtain that $(\ref{Alg})$ is equivalent to
\[
\Big(\frac{b}{a}\Big)^{\alpha-1}\geq 1+(b-a)\Big(1-\Big(\frac{b}{b+1}\Big)^{\alpha-1}\Big).
\]
This last inequality is obtained from the Bernoulli inequality $(1+x)^{r}\geq 1+rx$, for $r>0$ and $x>-1$, observing that 
\[
\Big(\frac{b}{a}\Big)^{\alpha-1}=\Big(1+\frac{b-a}{a}\Big)^{\alpha-1}\geq 1+(\alpha-1)\frac{b-a}{a}
\]
and
\[
\Big(\frac{b}{b+1}\Big)^{\alpha-1}=\Big(1-\frac{1}{b+1}\Big)^{\alpha-1}\geq 1-(\alpha-1)\frac{1}{b+1}.
\]
\qed

Define for all $\delta>0$, $\beta\geq 1$, $k\geq 3$ and $n\geq k+1$
\[
\Delta^n_k:=(n^{\beta}-(n-k)^{\beta}-2)^{-\delta}-((n+1)^{\beta}-(n-k)^{\beta})^{-\delta}.
\]

\begin{lemma}
\label{HJ1}
For all $\delta>0$, $\beta\geq 1$ and $k\geq 3$, $\Delta_k^n$
is a non increasing function of $n\geq k+1$.
\end{lemma}
\noindent
{\it Proof.}
The statement of the lemma is trivial for $\beta=1$.
For $\beta>1$, consider the function $f:[4,\infty)\to \R^+$ defined by
\[
f(x)=(x^{\beta}-(x-k)^{\beta}-2)^{-\delta}-((x+1)^{\beta}-(x-k)^{\beta})^{-\delta}.
\]
 In order to prove the result, it is enough to show that the derivative of $f$ is negative.
Since
\begin{align*}
f'(x)&=-\delta \beta\Big[\Big(x^{\beta}-(x-k)^{\beta}-2\Big)^{-\delta-1}\Big(x^{\beta-1}-(x-k)^{\beta-1}\Big)\nonumber\\
&\phantom{**}-\Big((x+1)^{\beta}-(x-k)^{\beta}\Big)^{-\delta-1}\Big((x+1)^{\beta-1}-(x-k)^{\beta-1}\Big)\Big]
\end{align*}
it is enough to show that 
\[
\frac{(x+1)^{\beta}-(x-k)^{\beta}}{x^{\beta}-(x-k)^{\beta}-2}\geq \frac{(x+1)^{\beta-1}-(x-k)^{\beta-1}}{x^{\beta-1}-(x-k)^{\beta-1}}.
\]
But this last inequality follows from Lemma \ref{Lemalg}.
\qed

\begin{lemma}
\label{HJ2}
For all $\delta>0$, $\beta\geq 1$ and $k\geq 3$, we have that 
\begin{equation*}
\Delta_k^{k+1}\leq 4\frac{2^{\beta} 4^{\delta}\beta}{k^{\delta\beta+1}}.
\end{equation*}
\end{lemma}
\noindent
{\it Proof.} 
Observe that for $k\geq 3$, 
\begin{align}
\label{Anap}
\Delta_k^{k+1}&=\int_{(k+1)^{\beta}-3}^{(k+2)^{\beta}-1}\frac{1}{x^{1+\delta}}dx \nonumber\\
& \leq \frac{(k+2)^{\beta}-(k+1)^{\beta}+2}{((k+1)^{\beta}-3)^{1+\delta}} \leq \frac{\beta (k+2)^{\beta-1}+2}{((k+1)^{\beta}-3)^{1+\delta}} \leq 4\frac{2^{\beta} 4^{\delta}\beta}{k^{\delta\beta+1}},
\end{align}
where to obtain the second inequality in (\ref{Anap}), we used the inequality, 
\[
(a+b)^{\alpha}\leq a^{\alpha}+\alpha b (a+b)^{\alpha-1}
\]
for $\alpha\geq 1$ and $a, b\geq 0$. This inequality can be obtained using the fundamental theorem of calculus to the function $f(x)=x^{\alpha}$.

To obtain the last inequality in (\ref{Anap}), we used that for $k\geq 3$, 
\[
\beta (k+2)^{\beta-1}+2\leq 2\beta(k+2)^{\beta-1}\leq 2^{\beta}\beta k^{\beta-1}
\]
and
\[
(k+1)^{\beta}-3=(k+1)^{\beta}\Big(1-\frac{3}{(k+1)^{\beta}}\Big)\geq \frac{(k+1)^{\beta}}{4}\geq \frac{k^{\beta}}{4}.
\]
\qed

Finally, we recall the following lemma in \cite{bressaud/fernandez/galves/1999a} (see also Lemma A.4 in \cite{giacomin2007random}) that gives an estimate for the renewal sequence that will appear in the proof of Theorem \ref{Maintheo}. We state the lemma using a notation that is adapted to our purpose.
\begin{lemma}[ Proposition 2 item (iv) in \cite{bressaud/fernandez/galves/1999a}]
\label{renewal}
Let $(f_k)_{k \geq 1}$ be a sequence of positive real numbers such that $\sum_{k=1}^\infty f_k < 1$. Supose that $(u_k)_{k \geq 1}$ is a sequence with $u_0 = 1$ and satisfies the renewal equation
\begin{equation*}
u_n = \sum_{k=1}^n f_k u_{n-k}.
\end{equation*}
If $f_n \leq c_1/n^{1+\alpha}$, for some $\alpha >0$ and positive constant $c_1$, then $u_n \leq c_2/n^{1+\alpha}$, where $c_2$ is a constant that depends on $(f_k)_{k \geq 1}$.
\end{lemma}

\section{Proof of Theorem \ref{Maintheo}}
\label{Lil}
 Let $x, y\in \mathcal{X}_-$. $\mu, \nu$ are compatible measures such that $\mu[\eta_{-\infty}^{0}\in\cdot\;]=\delta_x(\cdot)$ and $\nu[\eta_{-\infty}^{0}\in\cdot\;]=\delta_y(\cdot)$. 
We now construct the coupling of $\mu$ and $\nu$, that we call $\P^{x,y}$, as follows. We start by defining 
$$\P^{x,y}[\hat{\eta}_{-\infty}^{0}\in\cdot\;, \hat{\omega}_{-\infty}^{0}\in\cdot\;]=\delta_x \otimes\delta_y.$$ Then, for all $n\geq 1$,  given the pasts $\hat{\eta}_{-\infty}^{K_n-1}$ and  $\hat{\omega}_{-\infty}^{K_n-1}$, we maximally couple $\hat{\eta}_{K_n}^{K_{n+1}-1}$ and $\hat{\omega}_{K_n}^{K_{n+1}-1}$ to complete the construction of $\P^{x,y}$.

Next, we show that $\P^{x,y}$ satisfies the inequality in Theorem \ref{Maintheo}. For all $n\geq  1$ and $0\leq k\leq n-1$, define
\begin{equation*}
q^n_k=\sup_{x,y,a,b \in\mathcal{X}}\P^{x,y}[X_{n}=1\mid X_{n-k}^{n-1}=0, \hat{\eta}^{K_{n-k}-1}_{1} = a_{1}^{K_{n-k}-1}, \hat{\omega}^{K_{n-k}-1}_{1} = b_{1}^{K_{n-k}-1}],
\end{equation*}
with the convention that if $k>l$ then elements of the form $a_k^l$ are dropped from the conditional. The shorthand notation $X_{n-k}^{n-1}=0$ means that $X_{n-k} = 0, \ldots, X_{n-1} = 0$. Observe that for all $n\geq  1$ and $0\leq k\leq n-2$ we have $q^n_{k} \geq q^n_{k+1}$.

We start by proving the following
\begin{lemma}
\label{HJ0}
Suppose that $(\chi^2_n(\phi))_{j\geq 0}\in \ell^1$. Then, there exists $\eps>0$ such that for all $k\geq 0$ and all $n\geq k+1$,
\begin{equation}
\label{Bret}
q_k^n\leq\sqrt{1-\exp\left(-\sum_{j=0}^{\infty}\chi_j^2(\phi)\right)}\leq 1-\eps.
\end{equation}
For $k\geq 1$ and $n\geq k+1$, we also have
\begin{equation}
\label{Pins}
q_k^n\leq \sqrt{\frac{1}{2}\sum_{j=K_n}^{K_{n+1}-1}\chi^2_{j-K_{n-k}}(\phi)}.
\end{equation}
\end{lemma}
\noindent
{\it Proof.}
Inequality (\ref{Bret}) is a direct consequence of the Bretagnolle-Huber inequality (cf. \cite[eq. 4]{sason2016f}) and Lemma \ref{chideu}. Inequality (\ref{Pins}) is a direct consequence of the Pinsker inequality (cf. \cite[eq. 1]{sason2016f}) and again Lemma \ref{chideu}.
\qed

Now, on some probability space $(\Omega, \mathcal{F}, P)$, consider the random process $Y=(Y_n)_{n\geq 0}$ with values in $\{0,1\}$ such that $Y_0=1$ and  for $n\geq 1$ and $0\leq k\leq n-1$,
\begin{equation*}
P[Y_n=1\mid Y^{n-1}_{n-k}=0, Y_{n-k-1}=1, Y_1^{n-k-2}]=q^n_k.
\end{equation*}
For all $m\geq 1$ and $a,b \in \{0,1\}^{m}$ we say that $a\geq b$ if $a_i \geq b_i$ for $i \in \{1,\ldots, m\}$. By construction, for all $n\geq 2$, $a,b \in \{0,1\}^{n-1}$, and $a\geq b$, we have $P[Y_n = 1 | Y^{n-1}_1 = a] \geq \P^{x,y}[X_n = 1 | X^{n-1}_1 = b]$. Therefore, by applying Strassen's Theorem on stochastic domination \cite{lindvall1999strassen} inductively on $n$, we can construct a coupling measure $Q$ such that, for  $a,b \in \{0,1\}^{n-1}$, and $a\geq b$, we have $Q[Y_n \geq X_n| Y^{n-1}_1 = a, X^{n-1}_1 = b] = 1$. Therefore, for all $n\geq 1$, we have $Q[Y_n \geq X_n] = 1$, which implies that
\begin{equation}\label{dominance}
P[Y_n = 1] \geq \P^{x,y}[X_n = 1],
\end{equation}
 for all $n \geq 1$.

Now, consider the process $Z=(Z_n)_{n\geq 0}$ with values in $\{0,1\}$ such that $Z_0=1$ and  for $n\geq 1$ and $0\leq k\leq n-1$,
\begin{equation}
\label{GHJ1}
P[Z_n=1\mid Z^{n-1}_{n-k}=0, Z_{n-k-1}=1, Z_1^{n-k-2}]=b_k:=\sup_{n\geq k+1} q_k^n.
\end{equation}
Observe that, for all $k \geq 0$, we have $b_k \geq b_{k+1}$. Using the same argument used to show \eqref{dominance}, we have that $P[Z_n = 1] \geq P[Y_n = 1]$, for all $n\geq 1$.
Also, by (\ref{Pins}), and Lemmas \ref{HJ1} and \ref{HJ2}, we have  that for $k\geq 3$ and $n\geq k+1$,
\begin{align*}
2(q_k^n)^2&\leq \sum_{j=K_n}^{K_{n+1}-1}\chi^2_{j-K_{n-k}}(\phi)\\
&\leq C\sum_{j=\lfloor n^{\beta}\rfloor}^{\lfloor(n+1)^{\beta}\rfloor-1}\frac{1}{(j-\lfloor (n-k)^{\beta}\rfloor)^{1+\delta}}=C\sum_{j=\lfloor n^{\beta}\rfloor-\lfloor (n-k)^{\beta}\rfloor}^{\lfloor(n+1)^{\beta}\rfloor-\lfloor (n-k)^{\beta}\rfloor-1}\frac{1}{j^{1+\delta}}\\
&\leq C\int_{n^{\beta}-(n-k)^{\beta}-2}^{(n+1)^{\beta}-(n-k)^{\beta}}\frac{1}{x^{1+\delta}}dx=C\frac{\Delta_k^n}{\delta}\leq C\delta^{-1}\Delta_k^{k+1}\leq 4C\frac{2^{\beta} 4^{\delta}\beta\delta^{-1}}{k^{\delta\beta+1}}.
\end{align*}
Using (\ref{Bret}), we obtain that $b_k\leq (2C\frac{2^{\beta} 4^{\delta}\beta\delta^{-1}}{k^{\delta\beta+1}})^{1/2}\wedge (1-\eps)$ for all $k\geq 3$ and
$b_k\leq 1-\eps$ for $0\leq k\leq 2$.

Next, let $f_i:= b_{i-1}\prod_{k=0}^{i-2}(1-b_k)$
for $i\geq 1$ (with the convention that $\prod_{j=0}^{-1}=1$)
and
$u_i:= P[Z_i = 1]$
for $i \geq 0$. 
We have that
$$P[Z_n = 1] = \sum_{k= 1}^nP[Z_n = 1, Z^{n-1}_{n-k+1} = 0| Z_{n-k} = 1]P[Z_{n-k} = 1],$$
hence the following renewal equation holds
$$u_n = \sum_{k = 1}^nf_{k}u_{n-k}.$$
By definition, we have that $\sum_{k= 1}^\infty f_k = 1-\prod_{k = 1}^\infty(1-b_k)$. If $\beta>\delta^{-1}$, we have $\sum_{k = 0}^\infty b_k < \infty$. Hence $\sum_{k= 1}^\infty f_k < 1$. Moreover, when $\beta>\delta^{-1}$ we have that $b_k \leq c_1k^{-(\delta\beta+1)/2}$, for some positive constant $c_1$ that depends on $C, \delta$ and $\beta$. From Lemma \ref{renewal}, we have that, for all $n\geq 1$,
\begin{align*}
u_n\leq \frac{C_1}{n^{\frac{\delta\beta+1}{2}}}
\end{align*}
where $C_1$ is a positive constant that depends on $C, \delta$ and $\beta$.
Because $P[Z_n = 1] \geq \P^{x,y}[X_n = 1]$,  we obtain that for all $n\geq 1$,
\begin{equation*}
\P^{x,y}[X_n=1]\leq \frac{C_1}{n^{\frac{\delta\beta+1}{2}}}
\end{equation*}
for $\beta>\delta^{-1}$. Because the bound is uniform on $x,y \in \X_-$, we obtain the desired result.
\qed

\section{Proofs of Corollaries \ref{Cor1} and \ref{Cor2}}
\subsection{Proof of Corollary  \ref{Cor1}}

 For $k\in [K_n, K_{n+1})$ and all $y,z \in \X_-$, using the coupling inequality for total variation distance (cf. \cite{thorisson/2000}), we have that
\begin{equation*}
d_{\text{TV}}(\mu^y[\eta_k \in \cdot\;],\mu^z[\eta_k \in \cdot\;])\leq \P[\hat{\eta}_k\neq \hat{\omega}_k]\leq \P[X_n=1].
\end{equation*}
Then, by Theorem \ref{Maintheo}, we obtain that
\begin{equation*}
\P[\hat{\eta}_k\neq \hat{\omega}_k]\leq \frac{C_1}{n^{\frac{\beta \delta+1}{2}}}.
\end{equation*}
for all $\beta\geq 1$ and $\beta>\delta^{-1}$. If $\delta>1$, just take $\beta=1$. In this case $k=n$, thus we obtain Corollary \ref{Cor1} with constant $C_2$ that depends on $C$ and $\delta$. If $\delta\in (0,1)$, since $k\leq (n+1)^{\beta}$, we have $n\geq k^{1/\beta}-1$. This leads to
\begin{equation*}
\P[\hat{\eta}_k\neq \hat{\omega}_k]\leq \frac{C_3}{k^{\frac{\beta \delta+1}{2\beta}}}
\end{equation*}
for all $k\geq 1$, where $C_3$ is a positive constant that depends on $C, \delta$, and $\beta$. Now, observe that  for any $0< \delta'<\delta$, we can choose $\beta$ such that $\beta\geq 1$, $\beta>\delta^{-1}$ and $\frac{\beta \delta+1}{2\beta}\geq \delta'$.
\qed
\subsection{Proof of Corollary  \ref{Cor2}}

Consider $k\in [K_n, K_{n+1})$.  Let 
$$\theta=\inf\{n\geq 1: \hat{\eta}_k=\hat{\omega}_k,\; \text{for all}\; k\geq n\},$$ with the convention that $\inf\emptyset=\infty$. We start by observing that
\begin{align*}
\P[\theta>k]\leq \P\Big[\bigcup_{j\geq n} \{ X_j=1\}\Big]\leq \sum_{j\geq n} \P[X_j=1].
\end{align*}
By Theorem \ref{Maintheo}, we obtain that
\begin{align*}
\P[\theta>k]\leq C_1\sum_{j\geq n} \frac{1}{n^{\frac{\beta \delta+1}{2}}}\leq \frac{C_1'}{n^{\frac{\beta\delta-1}{2}}}
\end{align*}
for $\beta\geq 1$, $\beta>\delta^{-1}$ and $C_1'$ a positive constant that depends on $C, \delta$, and $\beta$. Since $n\geq k^{1/\beta}-1$, we obtain that
\begin{align*}
\P[\theta>k]\leq \frac{C_4}{k^{\frac{\beta\delta-1}{2\beta}}}
\end{align*}
for all $k\geq 1$ and $C_4$ a positive constant that depends on $C, \delta$, and $\beta$. Finally, notice that for all $\delta'<\delta$, we can choose $\beta$ large enough such that $\frac{\beta\delta-1}{2\beta}\geq \delta'$. Using the coupling inequality (cf. \cite{thorisson/2000}), we conclude that
\begin{align*}
M(n) := \sup_{y,z} d_{\text{TV}}(\mu^y[(\eta_j)_{j\geq n} \in \cdot\;],\mu^z[(\eta_j)_{j\geq n} \in \cdot\;]) \leq \P[\theta>k].
\end{align*}
\qed

\section{Proof of Theorem \ref{Maintheo2}}
Consider $K_m=\lfloor m^{\beta} \rfloor$, for $m\ge 1$. For each $x, y\in \mathcal{X}_-$, we consider a probability space $(\Omega, \mathcal{F}, \P^{x,y})$ that supports the random elements $\tilde{\eta}$, $\tilde{\omega}$ and $\tilde{Z}$ defined as follows. Let $\tilde{\eta}_*\P^{x,y}$, $\tilde{\omega}_*\P^{x,y}$ be compatible with $\phi$ and $\tilde{\eta}_{-\infty}^{0}=x,  \tilde{\omega}_{-\infty}^{0}=y$. Also, for all $m\geq 1$ given the pasts $\tilde{\eta}_{-\infty}^{K_m-1}$ and  $\tilde{\omega}_{-\infty}^{K_m-1}$, the blocks $\tilde{\eta}_{K_m}^{K_{m+1}-1}$ and $\tilde{\omega}_{K_m}^{K_{m+1}-1}$ are maximally coupled. Under $\P^{x,y}$, the process $\tilde{Z}$ has the same law as the process $Z$ defined in Section \ref{Lil} and verifies $\tilde{Z}_m\geq \tilde{X}_m:={\color{red}\mathds{1}}\{\exists j\in [K_m, K_{m+1}), \tilde{\eta}_j \neq \tilde{\omega}_j\}$ for all $m\geq 1$ (this is indeed possible since $Z$ stochastically dominates $X$, see Section \ref{Lil}). We denote by $\E^{x,y}$ the expectation with respect to $\P^{x,y}$.

Fix some $n\in \N$ and let $k$ be such that $n\in [K_{k-1}, K_{k})$. We will show that
\begin{equation}
\label{RIC}
\Big| \int f\circ T^n \;{\hat f} d\tilde\mu-\int f d\tilde\mu\int {\hat f} d\tilde\mu \Big|\leq c_1 P[Z_k=1]
\end{equation}
for some positive constant $c_1$ that depends only on $C$, $\delta$ and $\beta$.
From this point, Theorem \ref{Maintheo2} is easily obtained following the proof of Corollary \ref{Cor1}.
To obtain (\ref{RIC}), we follow the argument developed in \cite{bressaud/fernandez/galves/1999a}, Section 5.
Using (3.7) in  \cite{bressaud/fernandez/galves/1999a}, we first observe that
\begin{align*}
\Big| \int f\circ T^n \;{\hat f} d\tilde\mu-\int f d\tilde\mu\int {\hat f} d\tilde\mu\Big | \leq \|f\|_1\sup_{x,y} \E^{x,y}\Big[\big|\hat{f}(\tilde{\eta}_{-\infty}^n)-\hat{f}(\tilde{\omega}_{-\infty}^{n})  \big|\Big].                                                                                                                                      
\end{align*}
For $k\geq 1$, let 
\begin{equation*}
\theta_k=\inf\{0\leq m\leq k:\tilde{Z}_{k-m}=1\}.
\end{equation*}
We have that
\begin{align}  \label{eq:varbound}
\E^{x,y}\Big[\big|\hat{f}(\tilde{\eta}_{-\infty}^n)-\hat{f}(\tilde{\omega}_{-\infty}^{n})  \big|\Big]
&=\E^{x,y}\Big[\sum_{j=0}^k \mathds{1}\{\theta_k=j\}\big|\hat{f}(\tilde{\eta}_{-\infty}^n)-\hat{f}(\tilde{\omega}_{-\infty}^{n})  \big|\Big]\nonumber\\
&\leq \|\hat{f}\|_{\phi} \sum_{j=0}^k \var_{n-K_{k-j}+1}(e^{\phi})\P^{x,y}[\theta_k=j].
\end{align}
Observe that for all $0\leq j\leq k$,
\begin{align*}
\P^{x,y}[\theta_k=j]&=P[Z_k=0,\dots,Z_{k-j+1}=0,Z_{k-j}=1]\nonumber\\
&=\prod_{l=1}^j(1-b_{k-j+l})P[Z_{k-j}=1]
\end{align*}
where $b_{k-j+l}$ are from (\ref{GHJ1}).
Now, observe that for all $i\geq 1$ we have 
\begin{align*}
P[Z_i=1]&=\sum_{k=1}^ib_{i-k}P[Z_{i-1}=0\mid Z_k^{i-1}=0, Z_{k-1}=1]\nonumber\\
&=\sum_{k=1}^i b_{i-k}\prod_{l=1}^{i-k}(1-b_{i-k-l})P[Z_{k-1}=1].
\end{align*}
Thus, we have that for all $i\geq 1$
\begin{align*}
P[Z_i=1]&=\sum_{k=1}^if_kP[Z_{i-k}=1]
\end{align*}
with $f_k:=b_{k-1}\prod_{l=0}^{k-2}(1-b_l)$, $k\geq 1$.
From this, we obtain that
\begin{align*}
\E^{x,y}\Big[\big|\hat{f}(\tilde{\eta}_{-\infty}^n)-\hat{f}(\tilde{\omega}_{-\infty}^{n})  \big|\Big]&\leq 
\|\hat{f}\|_{\phi}\Bigg( \var_{n-K_{k}+1}(e^\phi)\sum_{l=1}^k f_lP[Z_{k-l}=1]\nonumber\\
&\;\;\;\;\;+ \sum_{l=1}^k \var_{n-K_{k-l}+1}(e^\phi)\prod_{m=1}^{l}(1-b_{k-l+m})P[Z_{k-l}=1]   \Bigg).
\end{align*}
We deduce that 
\begin{align*}
\sup_{x,y}\E^{x,y}\Big[\big|\hat{f}(\tilde{\eta}_{-\infty}^n)-\hat{f}(\tilde{\omega}_{-\infty}^{n})  \big|\Big]&\leq
 \kappa \sum_{l=1}^k f_lP[Z_{k-l}=1]=\kappa P[Z_{k}=1]
\end{align*}
with 
$$\kappa:=\var_{n-K_{k}+1}(e^\phi)+\sup_{1\leq l\leq k}\frac{\var_{n-K_{k-l}+1}(e^\phi)}{f_l}.
$$

Finally, since $\frac{\var_{n-K_{k-l}+1}(e^\phi)}{b_{l-1}}\leq 2$ and $\prod_{j=0}^{\infty}(1-b_j)>0$ (using that $b_0 < 1$ and $\sum_{j=0}^\infty b_j<\infty$), we observe that 
\begin{equation*}
\kappa\leq 1+\frac{\var_{n-K_{k-l}+1}(e^\phi)}{b_{l-1}\prod_{j=0}^{\infty}(1-b_j)}\leq c_2
\end{equation*}
for some positive constant $c_2$ depending on $C$, $\delta$, and $\beta$.
\qed

\subsection{Proof of Corollary  \ref{coro3}}
The potential $\phi$ is summable, therefore, there exists a normalized potential $\psi$ with the same unique equilibrium state as $\phi$ \cite[Theorem 3.2.]{walters/1975}. If  $\var_{k}(\phi)\leq \frac{C}{k^{\frac{3+ \delta}{2}}}$ then $\var_{k}(\psi)\leq \frac{C'}{k^{\frac{1+ \delta}{2}}}$ for some constat $C' > 0$ that depends only on $C$ (see for example \cite[Proposition 1]{pollicott2000rates}). Closely following the proof of Theorem \ref{Maintheo2} using the measures compatible with $\psi$, we obtain the desired results. The only difference is that in \eqref{eq:varbound} we use the bound 
\begin{equation*} 
\E^{x,y}\Big[\big|\hat{f}(\tilde{\eta}_{-\infty}^n)-\hat{f}(\tilde{\omega}_{-\infty}^{n})  \big|\Big]
\leq \|\hat{f}\|_{\phi} \sum_{j=0}^k \var_{n-K_{k-j}+1}(e^{\phi})\P^{x,y}[\theta_k=j],
\end{equation*}
instead of
\begin{equation*} 
\E^{x,y}\Big[\big|\hat{f}(\tilde{\eta}_{-\infty}^n)-\hat{f}(\tilde{\omega}_{-\infty}^{n})  \big|\Big]
\leq \|\hat{f}\|_{\psi} \sum_{j=0}^k \var_{n-K_{k-j}+1}(e^{\psi})\P^{x,y}[\theta_k=j].
\end{equation*}
\qed

\section{Applications}
\subsection{Functional central limit theorem (FCLT) for potentials with non-summable variations}
Let $\sigma > 0$. A function $h:A\to \R$ satisfies the FCLT, also called Weak Invariance Principle, when the process $\{\zeta_n(t), t \in [0,1], n\geq 1\}$ defined by 
\begin{equation*}
\zeta_n(t) = \frac{1}{\sigma \sqrt{n}} \sum_{i = 0}^{\lfloor nt \rfloor} h\circ \eta_i,
\end{equation*}
converges weakly to a standard Brownian motion on $D[0,1]$. \cite[Section 4.3]{tyran2005invariance} showed that the FCLT holds when the potential $\phi$ has summable variations. A straightforward application of Theorem 1 in \cite{tyran2005invariance} and our Corollary~ \ref{Cor1} is the following FCLT for potentials with non-summable variations.
\begin{prop}
\label{fclt}
Assume that the alphabet $A$ is finite and $\phi$ satisfies Assumption ($\mathcal{A}$) with $\delta \in (1/2,1]$. Let $\mu$ be shift-invariant and compatible with $\phi$. Also, let $h:A \to \R$ be a function such that $\int h\circ \eta_0\; d\mu = 0$. If $\sigma^2 :=\int (h\circ \eta_0)^2\; d\mu > 0$, then $h$ satisfies the FCLT. 
\end{prop}
\begin{proof}
Because alphabet is finite, we have that $\sigma^2 \leq \|h\|^2_\infty < \infty$ as required by Theorem 1 in \cite{tyran2005invariance}. It remains to verify the condition on the mixing rate. Using the processes $\{\tilde{\eta}_i, i \in \Z\}$ and $\{\tilde{\omega}_i, i \in \Z\}$ introduced in the proof of Theorem \ref{Maintheo2}, it is sufficient to check that there exists $\gamma > 1/2$ such that
\begin{equation}\label{condFCLT}
\limsup_{n \rightarrow \infty} n^\gamma \sup_{x,y}\E^{x,y}[|h(\tilde{\eta}_n)-h(\tilde{\omega}_n)|] < \infty.
\end{equation}
We have from Corollary ~ \ref{Cor1} that, for all $x,y \in \X_-$,
\begin{align*} 
\E^{x,y}[|h(\tilde{\eta}_n)-h(\tilde{\omega}_n)|] &\leq 2\|h\|_{\infty} \P^{x,y}[\tilde{\eta}_n\neq \tilde{\omega}_n]\\
&\leq \frac{c_1}{n^{\delta'}},
\end{align*}
where $\delta' < \delta$ and $c_1 $ is positive constant that depends on $C, \delta, \delta'$, and $h$. Taking $\delta'$ and $\gamma$ such that $1/2 < \gamma \leq \delta' <\delta$, we obtain \eqref{condFCLT}.
\end{proof}

\subsection{Hoeffding-type inequality for potentials with non-summable variations}

Hoeffding-type inequality gives finite sample bounds for deviations of additive functionals from their mean. When the variation rate of the potential is summable, we have an exponential inequality \cite{marton/1998, gallesco/gallo/takahashi/2014, chazottes2020optimal}. Nevertheless, the rate of concentration for potentials when the variation rate is not summable is an open question. Using our result, we can obtain the following stretched exponential inequality for sums of random variables.
\begin{prop}
\label{ConcIneq}
Assume that the alphabet $A$ is finite and $\phi$ satisfies Assumption ($\mathcal{A}$) with $\delta \in (1/2,1]$. Let $\mu$ be compatible with $\phi$. For all $\delta'<2\delta-1$, $n\geq 1$, $t\geq 0$ and all functions $h:A \to \R$ we have that
\begin{equation*}
\mu\Bigg[\Big|\frac{1}{n}\sum_{i=1}^n(h(\eta_i)-\E[h(\eta_i)])\Big|\geq t\Bigg]\leq 2\exp\Bigg\{-\frac{C_{10}n^{\delta'}t^2}{R(h)^2}\Bigg\}
\end{equation*}
where $R(h):=\max_{a\in A}h(a)-\min_{a\in A}h(a)$ and $C_{10}$  is a constant that depends on $C$, $\delta$ and~$\delta'$.
\end{prop}
\noindent
\begin{proof} This is a consequence of Theorem 1 of \cite{chazottes/collet/kulske/regig/2007} and Corollary~\ref{Cor1}. In order to apply Theorem 1 of \cite{chazottes/collet/kulske/regig/2007}, we need to estimate the terms $\|\overline{D}\|_{\ell^2(\N)}^2$ and $\|\delta f\|^2_{\ell^2(\N)}$ (for $f(x_1,\dots,x_n)=\frac{1}{n}\sum_{i=1}^n h(x_i)$) there.
We have that
\begin{equation*}
\|\overline{D}\|_{\ell^2(\N)}^2\leq \sup_{x,y\in \mathcal{X}_-} \Big(1+\sum_{i=1}^n\P^{x,y}[\hat{\eta}_i\neq \hat{\omega}_i]\Big)^2.
\end{equation*}
Now, for $\delta'<2\delta-1$, using Corollary~ \ref{Cor1}, we obtain that
\begin{equation}
\label{normD}
\|\overline{D}\|_{\ell^2(\N)}^2\leq c_1 n^{1-\delta'}
\end{equation}
for some positive constant $c_1$ depending on $C, \delta$ and $\delta'$.

For a given function $f: A^{n}\to \R$ we define the oscillation of $f$ at site $i\in \{1,\dots,n\}$, by 
\[
\delta_if:= \sup_{x_j=x'_j, j \neq i}\big|f(x_1,\dots,x_n)-f(x'_1,\dots,x'_n)\big|.
\]
Now, taking $f(x_1,\dots, x_{n})= \frac{1}{n}\sum_{i=1}^{n}h(x_i)$ we have $\delta_i f=\frac{R(h)}{n}$ for $i\in \{1,\dots,n\}$. Thus, we obtain that
\begin{align}
\label{normf}
\|\delta f\|^2_{\ell^2(\N)}=\sum_{i=1}^{n}\left(\frac{R(h)}{n}\right)^2= \frac{R(h)^2}{n}.
\end{align}
Finally, using (\ref{normD}) and (\ref{normf}) in Theorem 1 of~\cite{chazottes/collet/kulske/regig/2007}, we obtain Proposition \ref{ConcIneq}.
\end{proof}

\subsection{Poisson autoregression model}
\label{Poisson}
As a second application of our results, we consider a model with countable infinite alphabet called Poisson autoregression, which is popular in applications \cite{kedem2005regression}. Only the Markovian case of these models were studied in the literature. We will show how we can choose the parameters of non-Markovian Possion autoregression models to satisfy the Assumption ($\mathcal{A}$) and thus apply the results of Section~\ref{Results}. 

Consider an absolutely converging sequence $(\beta_i)_{i\geq 1}$ and a sequence of non negative integers $(\gamma_i)_{i\geq 1}$ such that $S:=\sum_{i=1}^{\infty}|\beta_i|\gamma_i<\infty$.
Consider $A=\Z_+$ and the potential $\phi$ defined for all $x\in \mathcal{X}_-$ by
$$
\phi(x)=-\lambda(x_{-\infty}^{-1}) + x_0\log \lambda(x_{-\infty}^{-1}) - \sum_{k = 0}^{x_0}\log(k),
$$
where
$$
\lambda(x_{-\infty}^{-1})=\exp\Big\{\sum_{i=1}^{\infty}\beta_i (x_{-i}\wedge \gamma_i)\Big\}.
$$
For this model, we obtain that
\begin{equation}
\label{RTYUP}
\chi^2_k(\phi)=\sup_{a\in \mathcal{X}} \sup_{x,y\in  \mathcal{X}_-}\Bigg(e^{\lambda(a_{-k}^{-1}y)\Big(\frac{\lambda(a_{-k}^{-1}x)}{\lambda(a_{-k}^{-1}y)}-1\Big)^2}-1\Bigg).
\end{equation}
Now, since $e^{-S}  \leq \lambda(x_{-\infty}^{-1})\leq  e^S$ and  the exponential function is locally bilipschitz, using (\ref{RTYUP}), we have that
\begin{equation*}
c_1^{-1}  \Big(\sum_{i=k+1}^{\infty}|\beta_i|\gamma_i\Big)^2 \leq  \chi^2_k(\phi)\leq c_1 \Big(\sum_{i=k+1}^{\infty}|\beta_i|\gamma_i\Big)^2
\end{equation*}
where $c_1$ is a positive constant that depends only on $S$. 

Finally, choosing the sequences $(\beta_i)_{i\geq 1}$ and $(\gamma_i)_{i\geq 1}$ such that 
$$\frac{c_2^{-1}}{i^{\frac{3+\varepsilon}{2}}}\leq |\beta_i|\gamma_i \leq \frac{c_2}{i^{\frac{3+\varepsilon}{2}}}$$
 for some $\varepsilon>0$ and $c_2\geq 1$, we obtain that  
\begin{equation*}
\frac{c_3^{-1}}{k^{1+\varepsilon}}\leq \chi^2_k(\phi)\leq \frac{c_3}{k^{1+\varepsilon}}.
\end{equation*}
where $c_3$ is a positive constant.

Finally, for this model, we mention that the existence of a shift-invariant probability measure compatible with $\phi$ is obtained by applying Theorem 5.1 of~\cite{johansson/oberg/pollicot/2007} with $K=e^{2\sinh S}$ and $\pi$ equal to the Poisson law with parameter $e^S$. Assumption ($\mathcal{A}$) implies the square summability of the variation, which guarantees the uniqueness of the shift-invariant probability measure \cite[Corollary 4.2.]{johansson/oberg/pollicot/2007}.

\section*{Acknowledgements}
\noindent
C.~Gallesco and D.Y.~Takahashi would like to thank Sandro Gallo for several fruitful discussions that motivated this work. We also thank Leandro Cioletti for comments on a early version of this manuscript.
C.~Gallesco was partially supported by FAPESP (grant 2017/19876-4) and CNPq (grant 312181/2017-5). D.Y.~Takahashi thanks the support of FAPESP  Research, Innovation and Dissemination Center for Neuromathematics (grant 2013/ 07699-0).

\bibliographystyle{alpha}
\bibliography{MixingTime}

\end{document}